\documentclass[11pt]{amsart} 

\usepackage{amsmath,amssymb,amsthm,mathtools}
\usepackage{mathrsfs}
\usepackage{hyperref}
\usepackage[margin=1.2in]{geometry}

\numberwithin{equation}{section}

\newtheorem{theorem}{Theorem}[section]
\newtheorem{proposition}[theorem]{Proposition}
\newtheorem{lemma}[theorem]{Lemma}
\newtheorem{corollary}[theorem]{Corollary}

\theoremstyle{definition}
\newtheorem{definition}[theorem]{Definition}
\newtheorem{remark}[theorem]{Remark}

\newcommand{\R}{\mathbb{R}}
\newcommand{\N}{\mathbb{N}}
\newcommand{\Z}{\mathbb{Z}}

\newcommand{\supp}{\operatorname{supp}}
\newcommand{\vol}{\operatorname{vol}}
\newcommand{\diam}{\operatorname{diam}}

\newcommand{\eps}{\varepsilon}

\begin{document}

\title[Kuznecov formulae]{Kuznecov formulae for fractal measures}

\author{Yakun Xi}
\address{School of Mathematical Sciences, Zhejiang University, Hangzhou 310027, China}
\email{yakunxi@zju.edu.cn}

\begin{abstract}
Let $(M,g)$ be a compact, connected Riemannian manifold of dimension $n\ge 2$, and let
$\{e_j\}_{j=0}^\infty$ be an orthonormal basis of Laplace eigenfunctions
$-\Delta_g e_j=\lambda_j^2 e_j$. Given a finite Borel measure $\mu$ on $M$, consider the
Kuznecov sum
\[
  N_\mu(\lambda):=\sum_{\lambda_j\le \lambda}\Bigl|\int_M e_j\,d\mu\Bigr|^2.
\]
Assume that $\mu$ admits an averaged
$s$-density constant $A_\mu$ with correlation dimension $s\in(0,n)$. We prove that
\[
  N_\mu(\lambda)
  = (2\pi)^{-(n-s)}\,\vol(B^{\,n-s})\,A_\mu\,\lambda^{n-s}
    + o(\lambda^{n-s})
  \qquad (\lambda\to\infty).
\]
The averaged $s$-density condition is necessary for such a one-term asymptotic, and
in general, the remainder $o(\lambda^{n-s})$ is sharp in the sense that it cannot be
improved uniformly to a power-saving error term. This extends the classical Kuznecov
formula of Zelditch \cite{Zelditch92} for smooth submanifold measures to a broad class
of singular and fractal measures.
\end{abstract}

\maketitle
\section{Introduction}

Let $(M,g)$ be a compact, connected, smooth Riemannian manifold of dimension
$n\ge 2$, and let
\[
  -\Delta_g e_j = \lambda_j^2 e_j,\qquad
  0 = \lambda_0 <\lambda_1 \le \lambda_2 \le \cdots,
\]
denote a complete orthonormal basis of real Laplace eigenfunctions.  Given a
finite Borel measure $\mu$ on $M$, we consider the associated Kuznecov sum
\begin{equation}\label{eq:Kuz-sum}
  N_\mu(\lambda) := \sum_{\lambda_j\le\lambda} \Big|\int_M e_j\,d\mu\Big|^2,
  \qquad \lambda\ge 1.
\end{equation}

A basic and well-studied case is when $\mu$ is the Riemannian volume measure
on an embedded submanifold.  Let $H\subset M$ be an embedded submanifold of
dimension $k$ with induced Riemannian metric and volume measure $dV_H$.  In
his seminal work \cite{Zelditch92}, Zelditch proved a Kuznecov sum formula
for these coefficients, showing that
\begin{equation}\label{eq:Zelditch}
  \sum_{\lambda_j\le\lambda}
  \biggl|\int_H e_j\,dV_H\biggr|^2
  = (2\pi)^{-(n-k)} \,\vol(B^{n-k})\,\vol(H)\,\lambda^{\,n-k}
    + O(\lambda^{n-k-1}),
\end{equation}
where $B^{n-k}\subset\R^{n-k}$ is the Euclidean unit ball and
$\vol(H)$ denotes the $k$-dimensional Riemannian volume of $H$.  Thus the
leading growth of the Kuznecov sum is governed by the ambient dimension $n$,
the submanifold dimension $k$, and an explicit geometric factor.

More recently, Wyman and the author \cite{WymanXiKuz} refined Zelditch's formula by
obtaining a two-term asymptotic for \eqref{eq:Zelditch}, identifying the
next-order oscillatory term.  In these works the measure is always smooth and
supported on a submanifold.  On the other hand, there has been growing recent interest
in the behavior of eigenfunctions restricted to general Borel sets.  For example,
Eswarathasan and Pramanik \cite{EswarathasanPramanik22} initiated a program to study
the $L^p$ norms of eigenfunctions with respect to a fractal measure $\mu$, connecting
classical results of Sogge \cite{sogge1988concerning} and Burq--G\'erard--Tzvetkov
\cite{BurqGerardTzvetkov07}.  An essentially sharp version of such bounds in two
dimensions was recently obtained by Gao, Miao, and the author \cite{GaoMiaoXi24}.
The goal of this paper is to extend Zelditch's Kuznecov formulae to a
much broader class of measures, including singular and fractal measures, while
retaining the same leading term.

Throughout the paper we use the following convention for Euclidean ball
volumes.  For $s\in \mathbb R^+$, we write $\vol(B^s)$ for
the “$s$-dimensional volume’’ of the Euclidean unit ball in formal dimension
$s$, defined by
\[
  \vol(B^s)
  := \frac{\pi^{s/2}}{\Gamma\bigl(\frac{s}{2}+1\bigr)}.
\]
When $s\in\Z^+$, this coincides with the usual Lebesgue $s$-volume of the unit
ball in $\R^s$.  For an embedded $k$-dimensional submanifold $H\subset M$ we
write $\vol(H)$ for its $k$-dimensional Riemannian volume.

We now describe the small-scale hypotheses on $\mu$ that will play the role of
``dimension'' and ``volume'' in our setting. Throughout, $B(x,r)$ denotes the
geodesic ball of radius $r$ centered at $x\in M$ with respect to the Riemannian
distance $d_g$.

For $r\ge 0$ define the distance distribution
\begin{equation}\label{eq:distance-distribution}
  F(r)
  :=(\mu\times\mu)\bigl(\{(x,y)\in M\times M:\ d_g(x,y)\le r\}\bigr)
  =\int_M \mu\bigl(B(x,r)\bigr)\,d\mu(x).
\end{equation}
The small-scale growth rate of $F(r)$ encodes the pair statistics of $\mu$ and
leads naturally to the \emph{correlation dimension} of $\mu$, see for instance
\cite{MattilaMoranRey00}.

\begin{definition}[Correlation dimension]\label{def:corr-dim}
Let $\mu$ be a finite Borel measure on $M$. If the following limit exists, we
define the \emph{correlation dimension} of $\mu$ to be
\begin{equation}\label{eq:corr-dim}
  \dim_C(\mu):=\lim_{r\to0^+}\frac{\log F(r)}{\log r}.
\end{equation}

\end{definition}

To obtain a one-term asymptotic for \eqref{eq:Kuz-sum}, it is necessary to assume
that the small-scale behavior of $F(r)$ stabilizes in the following sense.

\begin{definition}[Averaged $s$-density]\label{def:DD}
Let $s\in(0,n)$. We say that $\mu$ admits an \emph{averaged $s$-density} constant $A_\mu\in(0,\infty)$ if \begin{equation}\label{eq:DD}
  F(r)=\vol(B^s)\,A_\mu\,r^s+o(r^s)\qquad (r\to0^+).
\end{equation}
Equivalently,
\begin{equation}\label{eq:A-mu}
  A_\mu
  =\lim_{r\to0^+}\frac{F(r)}{\vol(B^s)\,r^s}
  =\lim_{r\to0^+}\frac{1}{\vol(B^s)\,r^s}\int_M \mu(B(x,r))\,d\mu(x).
\end{equation}
\end{definition}

When \eqref{eq:DD} holds, it forces $\dim_C(\mu)=s$, and $A_\mu$ is the
corresponding averaged $s$-dimensional density constant arising from the distance 
distribution.

We can now state our main result. It shows that once the distance distribution
$F(r)$ admits a stabilized small-scale asymptotic in the sense of \eqref{eq:DD},
the Kuznecov sum \eqref{eq:Kuz-sum} has a one-term asymptotic with exponent $n-s$
and an explicit leading constant.

\begin{theorem}\label{thm:main}
Let $(M,g)$ be a compact, connected Riemannian manifold of dimension $n\ge 2$.
Fix $s\in(0,n)$ and let $\mu$ be a finite Borel measure on $M$ admitting an
averaged $s$-density constant $A_\mu$ in the sense of \eqref{eq:DD}. Then
\begin{equation}\label{eq:Kuz-main}
  N_\mu(\lambda)
  =C_{n,s}\,A_\mu\,\lambda^{\,n-s}
    + o(\lambda^{\,n-s})
  \qquad (\lambda\to\infty),
\end{equation}
where
\[
  C_{n,s}=(2\pi)^{-(n-s)}\,\vol(B^{\,n-s}).
\]
\end{theorem}
The proof combines a heat-kernel regularization of the Kuznecov sum with a
Tauberian argument, with the leading term determined by the small-scale distance
distribution of $\mu$.
Combining \eqref{eq:Kuz-main} with Weyl's law for the eigenvalue
counting function shows that the Fourier coefficients $\int_M e_j\,d\mu$ have a
quantitative mean-square decay on average. The precise statement is recorded
in the following corollary.

\begin{corollary}[Average Fourier decay]\label{cor:avg-coeff}
Under the hypotheses of Theorem \ref{thm:main}, we have
\[
  \frac{1}{\#\{j:\lambda_j\le\lambda\}}
  \sum_{\lambda_j\le\lambda}\Bigl|\int_M e_j\,d\mu\Bigr|^2
  \approx \lambda^{-s},
  \qquad (\lambda\to\infty).
\]
\end{corollary}

If one wishes to interpret the exponent $s$ more geometrically, a natural
setting is the class of Ahlfors--David regular measures. Recall that a finite
Borel measure $\mu$ on $M$ is \emph{$s$-Ahlfors--David regular} if there exist
constants $c,C>0$ and $r_0>0$ such that
\begin{equation}\label{eq:Ahlfors}
  c\,r^s \le \mu(B(x,r)) \le C\,r^s,
  \qquad x\in\supp\mu,\ 0<r\le r_0.
\end{equation}
For such measures it is standard that the Hausdorff dimension of $\mu$ equals
$s$. Thus, within the Ahlfors--David regular class, the exponent in
\eqref{eq:Kuz-main} may be written as $n-\dim_H(\mu)$.

\begin{remark}\label{rem:density-implies-DD}
The averaged $s$-density condition \eqref{eq:DD} is automatically satisfied under
a stronger pointwise density hypothesis, together with a uniform upper $s$-ball
condition. Suppose that there exist $C>0$ and $r_0>0$ such that
$\mu(B(x,r))\le C r^s$ for all $x\in\supp\mu$ and $0<r\le r_0$, and that the
pointwise $s$-density
\[
  \theta(x)=\lim_{r\to0^+}\frac{\mu(B(x,r))}{\vol(B^s)\,r^s}
\]
exists for $\mu$-a.e.\ $x$. Then dominated convergence implies
\[
  \frac{F(r)}{\vol(B^s)\,r^s}
  =\int_M \frac{\mu(B(x,r))}{\vol(B^s)\,r^s}\,d\mu(x)
  \longrightarrow \int_M \theta(x)\,d\mu(x),
  \qquad (r\to0^+),
\]
so \eqref{eq:DD} holds with $A_\mu=\int_M \theta\,d\mu$. In particular,
\eqref{eq:DD} is strictly weaker than the existence of a pointwise $s$-density.
Indeed, by a theorem of Preiss \cite{Preiss87}, if a (Radon) measure has an
$s$-density $\theta$ with $0<\theta(x)<\infty$ for $\mu$-a.e.\ $x$, then
necessarily $s\in\N$ and $\mu$ is $s$-rectifiable. This integer-dimensional
rigidity does not apply to the averaged density condition \eqref{eq:DD}, which
allows genuinely non-integer $s$. See Proposition \ref{prop:nonlattice-example}.
\end{remark}

In the special case where $\mu=dV_H$ is the Riemannian volume measure on an
embedded submanifold $H\subset M$ of dimension $s$, $\mu$ admits a pointwise
$s$-density $\theta$ and hence satisfies the averaged $s$-density condition
\eqref{eq:DD} by Remark \ref{rem:density-implies-DD}. In fact $\theta\equiv 1$
and $A_\mu=\vol(H)$, so \eqref{eq:Kuz-main} reduces to
\[
  \sum_{\lambda_j\le\lambda}
  \biggl|\int_H e_j\,dV_H\biggr|^2
  = (2\pi)^{-(n-s)}\,\vol(B^{\,n-s})\,\vol(H)\,\lambda^{\,n-s}
    + o(\lambda^{\,n-s}),
\]
whose leading term agrees with Zelditch's Kuznecov formula \eqref{eq:Zelditch}
with $k=s$.

More generally, Theorem \ref{thm:main} implies that the same conclusion holds
when $H$ has only weak regularity. For instance, suppose that $H\subset M$ is a
compact embedded Lipschitz $s$-dimensional submanifold with Lipschitz boundary.
Let $\mu$ be the $s$-dimensional Hausdorff measure induced by $d_g$ on $H$. Then
$\mu$ admits a pointwise $s$-density $\theta\equiv 1$ for $\mu$-a.e.\ $x\in H$. Hence $\mu$ satisfies
\eqref{eq:DD} and Theorem \ref{thm:main} yields the same leading asymptotic with
$A_\mu=\vol(H)$. This provides a Kuznecov asymptotic at Lipschitz regularity,
whereas the classical results for submanifold measures assume that $H$ is smooth
and boundaryless.

In summary, Theorem \ref{thm:main} extends the classical Kuznecov formulae to a
broad class of singular and fractal measures. To the best of our knowledge,
Theorem \ref{thm:main} provides the first Kuznecov-type asymptotic for fractal
measures.

Moreover, as discussed in \cite{wyman2023can}, the left-hand side of Zelditch's
Kuznecov formula \eqref{eq:Zelditch} can be viewed as ``audible'' spectral
information associated with the submanifold $H$, while the right-hand side
encodes geometric information about $H$, namely its dimension $k$ and its
$k$-dimensional Riemannian volume $\vol(H)$. Theorem \ref{thm:main} extends this
audible--geometric correspondence to measures $\mu$ whose distance distribution
admits an averaged $s$-density constant $A_\mu$ in the sense of \eqref{eq:DD}. In
this case the correlation dimension $s$ and the effective ``volume'' $A_\mu$ are
audible through the leading asymptotics of $N_\mu(\lambda)$. If, in addition,
$\mu$ is $s$-Ahlfors--David regular, then
its Hausdorff dimension is also audible.

\smallskip

For $0<u<n$ let
\begin{equation}\label{eq:Riesz-energy}
  I_u(\mu) := \iint d_g(x,y)^{-u}\,d\mu(x)\,d\mu(y)
\end{equation}
denote the Riesz energy of $\mu$ with respect to the Riemannian distance
$d_g(x,y)$ on $M$. There is an earlier work of Hare and Roginskaya
\cite{HareRoginskaya03} which is closely related in spirit to our results. They
established energy formulae on compact Riemannian manifolds which relate
$I_u(\mu)$ to weighted sums of the projections of $\mu$ onto Laplace eigenspaces.
In Section \ref{sec:HR} we show that their identities imply the following
general upper bound, without assuming the existence of an averaged $s$-density
constant.

\begin{theorem}\label{thm:HR-bound}
Let $(M,g)$ be a compact Riemannian manifold of dimension $n\ge2$, and let
$\mu$ be a finite Borel measure on $M$. Fix $u\in(0,n)$ and suppose that the
Riesz energy $I_u(\mu)$ defined in \eqref{eq:Riesz-energy} is finite.
Then there exists a constant $C_1=C_1(M,g,u,\mu)$ such that
\[
  \sum_{\lambda_j\le\lambda} \Bigl|\int_M e_j\,d\mu\Bigr|^2
  \le C_1\,\lambda^{\,n-u}
  \qquad (\lambda\ge1).
\]
In particular, if $\mu$ has correlation dimension $\dim_C(\mu)=s$ for some
$s\in(0,n)$, then $I_u(\mu)<\infty$ for every $u<s$, and hence for each
$\varepsilon>0$ there exists a constant $C_2=C_2(M,g,s,\varepsilon,\mu)$ such
that
\[
  \sum_{\lambda_j\le\lambda} \Bigl|\int_M e_j\,d\mu\Bigr|^2
  \le C_2\,\lambda^{\,n-s+\varepsilon}
  \qquad (\lambda\ge1).
\]
\end{theorem}

Thus, for measures with finite Riesz energy, one also has nontrivial
polynomial upper bounds on the Kuznecov sum.  However, in the formulae of
Hare--Roginskaya the exponent $u$ is fixed, and the constants appearing in
Theorem \ref{thm:HR-bound} are not controlled uniformly as $u\to s^-$.  Since
our proof of Theorem \ref{thm:main} requires a precise description of the
singular behavior of $I_u(\mu)$ at $u=s$ in order to recover the leading
constant in \eqref{eq:Kuz-main}, one cannot simply invoke their identities
and pass to the limit $u\to s$.

\subsection*{Sharpness}\label{subsec:intro-sharpness}

We record two statements, proved in Section \ref{sec:sharpness}, which together
explain why the averaged $s$-density hypothesis \eqref{eq:DD} in
Theorem \ref{thm:main} is essentially optimal for a one-term asymptotic, and why
the remainder $o(\lambda^{n-s})$ in \eqref{eq:Kuz-main} cannot be improved
uniformly.

Even for Ahlfors--David regular measures, a genuine one-term asymptotic requires
stabilization of the small-scale geometry of $\mu$. The next proposition shows
that without \eqref{eq:DD}, a one-term asymptotic need not hold.

\begin{proposition}[Necessity of averaged $s$-density]
\label{prop:need-DD}
There exists an $s$-Ahlfors--David regular measure $\mu$ for which the quantity
$N_\mu(\lambda)/\lambda^{n-s}$ fails to converge as $\lambda\to\infty$.
\end{proposition}

Beyond these structural issues, we also show that the remainder
$o(\lambda^{n-s})$ in \eqref{eq:Kuz-main} is essentially optimal at this level of
generality: there is no uniform power-saving improvement valid even over the
Ahlfors--David regular subclass.

\begin{proposition}
    [Sharpness of the remainder]
\label{thm:no-uniform-ps}
Fix $\delta\in(0,1)$. There exists an $s$-Ahlfors--David regular measure $\mu$
admitting an averaged $s$-density constant $A_\mu$ such that
\[
  N_\mu(\lambda)
  =C_{n,s}A_\mu\,\lambda^{n-s}+O(\lambda^{n-s-\delta})
\]
fails to hold.
\end{proposition}

\subsection*{Organization}
In Section \ref{sec:heat} we introduce a heat-kernel regularization of the
Kuznecov sum and establish its small-time asymptotics in terms of the averaged
$s$-density constant $A_\mu$. Section \ref{sec:Tauberian} interprets this
smoothed Kuznecov sum as a Laplace--Stieltjes transform and applies a Tauberian
theorem to prove Theorem \ref{thm:main} and identify the constant $C_{n,s}$.
In Section \ref{sec:HR} we revisit the energy identities of Hare and Roginskaya
to derive general polynomial upper bounds for $N_\mu(\lambda)$, proving
Theorem \ref{thm:HR-bound}. Finally, Section \ref{sec:sharpness} is devoted to
sharpness and examples: we give examples showing that the averaged $s$-density hypothesis in
Theorem \ref{thm:main} is essentially optimal for a genuine one-term asymptotic,
and that the $o(\lambda^{n-s})$ remainder cannot, in general, be improved
uniformly to a power-saving error term.

\subsection*{Acknowledgements} This project is supported by the National Key R\&D Program of China under Grant No. 2022YFA1007200, the Natural Science Foundation of China under Grant No. 12571107, and the Zhejiang Provincial Natural Science Foundation of China under Grant No. LR25A010001. The author thanks Shi-Lei Kong for some helpful discussions. The author thanks Xucheng Hu and Diankun Liu for carefully reading an earlier draft of this paper.

\subsection*{Notation}
We write $A\lesssim B$ to mean that there exists an implicit constant $C>0$ such that
\[
  A\le C\,B.
\]
 We write
$A\gtrsim B$ if $B\lesssim A$, and $A\approx B$ if both $A\lesssim B$ and $B\lesssim A$.

For nonnegative functions $f,g$ and  $x_0\in\mathbb R\cup\{\infty\}$ we write
\[
  f(x)\sim g(x)\qquad (x\to x_0)
\]
to mean
\[
  \lim_{x\to x_0}\frac{f(x)}{g(x)}=1.
\]

\section{Heat kernel and a smoothed Kuznecov sum}\label{sec:heat}

In this section, we introduce a heat-kernel-regularized Kuznecov sum and
identify its small-time behavior in terms of the averaged $s$-density constant $A_\mu$. 

\subsection{Heat kernel asymptotics near the diagonal}

Throughout we write $p_t(x,y)$ for the heat kernel on $(M,g)$:
\[
  (\partial_t - \Delta_g)p_t(x,y) = 0,\qquad
  \lim_{t\to0^+}p_t(\,\cdot\,,y) = \delta_y,
\]
so that for every $f\in C^\infty(M)$,
\[
  e^{t\Delta_g}f(x)
  = \int_M p_t(x,y)f(y)\,dV(y)
  = \sum_{j=0}^\infty e^{-t\lambda_j^2}\langle f,e_j\rangle e_j(x).
\]

The small-time behavior of $p_t$ near the diagonal is classical. See, for
example, \cite[Ch. 2]{BGV92} or \cite[Ch. 8]{Grigoryan09}.

\begin{lemma}\label{lem:heat-local-expansion}
There exist $r_0>0$, $t_0>0$, and smooth functions $u_k(x,y)$ on
$\{(x,y)\in M\times M: d_g(x,y)< 2r_0\}$, $k=0,1,2,\dots$, with
$u_0(x,x)\equiv1$, such that for every $N\in\Z^+$ we have
\begin{equation}\label{eq:heat-local-expansion}
  p_t(x,y)
  = (4\pi t)^{-n/2} e^{-d_g(x,y)^2/(4t)}
  \biggl(\sum_{k=0}^{N-1} t^k u_k(x,y)\biggr)
  + R_N(t,x,y),
\end{equation}
and
\[
  |R_N(t,x,y)|
  \le C_N\,t^{N-n/2}
\]
for all $0<t\le t_0$ and all $x,y\in M$ with $d_g(x,y)<r_0$.
Moreover, there exist constants $c,C>0$ such that
\begin{equation}\label{eq:heat-off-diagonal}
  0\le p_t(x,y) \le C\,t^{-n/2}\exp\Bigl(-\frac{d_g(x,y)^2}{ct}\Bigr),
  \qquad t\in(0,1],\ x,y\in M.
\end{equation}
\end{lemma}

Since $u_0$ is smooth and $u_0(x,x)=1$, there exists $C>0$ such that
\begin{equation}\label{eq:u0-Lipschitz}
  |u_0(x,y)-1|\le C\,d_g(x,y)
\end{equation}
whenever $d_g(x,y)<r_0$.

\subsection{A smoothed Kuznecov sum}

Define the heat-regularized Kuznecov sum associated with $\mu$ by
\begin{equation}\label{eq:heat-Kuz}
  H_\mu(t)
  := \sum_{j=0}^\infty e^{-t\lambda_j^2}
       \Bigl|\int_M e_j\,d\mu\Bigr|^2,
  \qquad t>0.
\end{equation}
We first relate $H_\mu(t)$ to the heat kernel.

\begin{lemma}\label{lem:heat-spectral}
For every $t>0$,
\begin{equation}\label{eq:Hmu-heat-kernel}
  H_\mu(t)
  = \iint_{M\times M} p_t(x,y)\,d\mu(x)\,d\mu(y).
\end{equation}
\end{lemma}

\begin{proof}
Let
\[
  \alpha_j := \int_M e_j\,d\mu,
  \qquad j=0,1,2,\dots.
\]
Recall that the heat kernel satisfies
\[
  p_t(x,y)
  = \sum_{j=0}^\infty e^{-t\lambda_j^2}e_j(x)e_j(y).
\]
For each $J\in\mathbb N$ set
\[
  p_t^{(J)}(x,y)
  := \sum_{j=0}^J e^{-t\lambda_j^2}e_j(x)e_j(y),
\qquad
  S_J(t) := \sum_{j=0}^J e^{-t\lambda_j^2}|\alpha_j|^2.
\]
Since $e^{-t\lambda_j^2}\ge0$ and $|\alpha_j|^2\ge0$, the sequence $S_J(t)$ is
nondecreasing in $J$.

For each fixed $t>0$, the series converges absolutely and uniformly on $M\times M$, since $\sup_{x,y}|e_j(x)e_j(y)|\le \|e_j\|_\infty^2\lesssim \lambda_j^{n-1}$ by H\"ormander \cite{Hormander1968SpectralFunction} and $\sum_{j\ge0} e^{-t\lambda_j^2}\lambda_j^{n-1}<\infty$ by Weyl law.
In particular,
\[
  \sup_{x,y\in M}\bigl|p_t(x,y)-p_t^{(J)}(x,y)\bigr|\longrightarrow 0
  \qquad(J\to\infty),
\]
and $p_t$ is continuous and bounded on $M\times M$.

For each $J$ we have
\[
\begin{split}
  \iint_{M\times M} p_t^{(J)}(x,y)\,d\mu(x)\,d\mu(y)
  &= \iint_{M\times M}
     \sum_{j=0}^J e^{-t\lambda_j^2}e_j(x)e_j(y)\,d\mu(x)\,d\mu(y) \\
  &= \sum_{j=0}^J e^{-t\lambda_j^2}
      \biggl(\int_M e_j(x)\,d\mu(x)\biggr)
      \biggl(\int_M e_j(y)\,d\mu(y)\biggr) \\
  &= S_J(t),
\end{split}
\]
where we have interchanged the finite sum and the integrals.

Since $p_t^{(J)}\to p_t$ uniformly and $\mu$ is finite, we have
\[
  \lim_{J\to\infty}
  \iint_{M\times M} p_t^{(J)}(x,y)\,d\mu(x)\,d\mu(y)
  = \iint_{M\times M} p_t(x,y)\,d\mu(x)\,d\mu(y).
\]
Therefore $S_J(t)$ converges as $J\to\infty$, and
\[
  \sum_{j=0}^\infty e^{-t\lambda_j^2}|\alpha_j|^2
  = \lim_{J\to\infty} S_J(t)
  = \iint_{M\times M} p_t(x,y)\,d\mu(x)\,d\mu(y),
\]
which is precisely \eqref{eq:Hmu-heat-kernel}.
\end{proof}

\subsection{A Gaussian average}

Using the distance distribution $F(r)$ introduced above, we consider the
Gaussian average
\[
  G(t)
  := \iint e^{-d_g(x,y)^2/(4t)}\,d\mu(x)\,d\mu(y).
\]
Let $R(x,y):=d_g(x,y)$ and let $\nu:=R_*(\mu\times\mu)$ be the pushforward
of $\mu\times\mu$ under $R$.  Then $\nu$ is a finite Borel measure on
$[0,\diam(M)]$ and
\[
  F(r) = (\mu\times\mu)\bigl(\{(x,y): d_g(x,y)\le r\}\bigr)
       = \nu([0,r]),
\]
so $F$ is the distribution function of $\nu$.  In particular, for the
nonnegative Borel function $\varphi(r):=e^{-r^2/(4t)}$ we have
\[
  G(t)
  = \int_{M\times M} \varphi(R(x,y))\,d\mu(x)\,d\mu(y)
  = \int_{[0,\diam(M)]} \varphi(r)\,d\nu(r)
  = \int_0^{\diam(M)} e^{-r^2/(4t)}\,dF(r).
\]

The next lemma identifies the behavior of $G(t)$ as $t\to0^+$.

\begin{lemma}\label{lem:Gaussian-F}
Assume that $\mu$ satisfies \eqref{eq:DD}.
Then, as $t\to0^+$,
\begin{equation}\label{eq:G-asympt}
  G(t)
  = \gamma_s\,A_\mu\,t^{s/2} + o(t^{s/2}),
\end{equation}
where
\begin{equation}\label{eq:gamma-s}
  \gamma_s := s\,2^{\,s-1}\Gamma\Bigl(\frac{s}{2}\Bigr)\,\vol(B^s).
\end{equation}
In particular, there exists $C>0$ and $t_1>0$ such that
\begin{equation}\label{eq:G-upper}
  G(t)\le C\,t^{s/2}\qquad(0<t\le t_1).
\end{equation}
\end{lemma}

\begin{proof}
Write, by our assumption on $\mu$,
\[
  F(r) = \vol(B^s)A_\mu r^s + \mathcal E(r),
\]
with $\mathcal E(r)=o(r^s)$ as $r\to0^+$.  Fix $r_0>0$ to be a small constant.  Split
\[
  G(t)
  = \int_0^{r_0} e^{-r^2/(4t)}\,dF(r)
    + \int_{r_0}^{\diam(M)} e^{-r^2/(4t)}\,dF(r)
  =: G_0(t)+G_1(t).
\]
Since $F$ is bounded on $[r_0,\diam(M)]$, we have
\[
  |G_1(t)| \le F(\diam(M))\,e^{-r_0^2/(4t)} = o(t^N)
\]
for every $N>0$ as $t\to0^+$.  In particular $G_1(t)=o(t^{s/2})$.

For $G_0(t)$, decompose
\[
  G_0(t)
  = \vol(B^s)A_\mu\int_0^{r_0} e^{-r^2/(4t)}\,d(r^s)
    + \int_0^{r_0} e^{-r^2/(4t)}\,d\mathcal E(r)
  =: J_1(t)+J_2(t).
\]

For $J_1(t)$ we compute explicitly:
\[
  J_1(t)
  = \vol(B^s)A_\mu s\int_0^{r_0} r^{s-1}e^{-r^2/(4t)}\,dr
  = \vol(B^s)A_\mu s\int_0^\infty r^{s-1}e^{-r^2/(4t)}\,dr + O(e^{-r_0^2/(4t)}).
\]
The error term is $o(t^{s/2})$ as above.  Using the standard integral
\[
  \int_0^\infty r^{s-1}e^{-r^2/(4t)}\,dr
  = 2^{\,s-1}\Gamma\Bigl(\frac{s}{2}\Bigr)t^{s/2},
\]
we obtain
\[
  J_1(t)
  = \gamma_s\,A_\mu\,t^{s/2} + o(t^{s/2})
\]
with $\gamma_s$ as in \eqref{eq:gamma-s}.

For $J_2(t)$ we integrate by parts on $[0,r_0]$:
\[
  J_2(t)
  = e^{-r^2/(4t)}\mathcal E(r)\Big|_{0}^{r_0}
    + \int_0^{r_0}\frac{r}{2t}e^{-r^2/(4t)}\mathcal E(r)\,dr
  =: B(t)+I(t).
\]
Since $F(0)=0$ and hence $\mathcal E(0)=0$, we have
$B(t)=e^{-r_0^2/(4t)}\mathcal E(r_0)=o(t^{s/2})$.

To handle $I(t)$, fix $\eps>0$ and choose $\delta\in(0,r_0)$ so small that
$|\mathcal E(r)|\le\eps r^s$ for $0<r\le\delta$.  Split
\[
  I(t)
  = \int_0^{\delta}\frac{r}{2t}e^{-r^2/(4t)}\mathcal E(r)\,dr
    + \int_{\delta}^{r_0}\frac{r}{2t}e^{-r^2/(4t)}\mathcal E(r)\,dr
  =: I_1(t)+I_2(t).
\]

On $(0,\delta]$ we have $|\mathcal E(r)|\le\eps r^s$, so
\[
  |I_1(t)|
  \le \frac{\eps}{2t}\int_0^{\delta} r^{s+1}e^{-r^2/(4t)}\,dr
  \le \frac{\eps}{2t}\int_0^{\infty} r^{s+1}e^{-r^2/(4t)}\,dr
  = C_1\,\eps\,t^{s/2}
\]
for some $C_1$ depending only on $s$.

On $[\delta,r_0]$ we use the upper bound
$F(r)\le C r^s$ to deduce that $|\mathcal E(r)|\le C' r^s$ on this interval.
Thus
\[
  |I_2(t)|
  \le \frac{C'}{2t}\int_{\delta}^{r_0} r^{s+1}e^{-r^2/(4t)}\,dr
  \le C'' t^{-1}e^{-\delta^2/(4t)},
\]
which is $o(t^{s/2})$ as $t\to0^+$.  Hence $I(t)=o(t^{s/2})$, and since
$\eps>0$ was arbitrary, we also have $J_2(t)=B(t)+I(t)=o(t^{s/2})$.

Combining the estimates for $J_1(t)$ and $J_2(t)$ gives
\eqref{eq:G-asympt}.  The upper bound \eqref{eq:G-upper} follows from
\eqref{eq:G-asympt} by taking $t$ sufficiently small and then enlarging the
constant $C$ to cover all $t\in(0,t_1]$.
\end{proof}

\subsection{Small-time asymptotics of the smoothed Kuznecov sum}

We can now identify the leading behavior of $H_\mu(t)$ as $t\to0^+$.

\begin{lemma}\label{lem:Hmu-asympt}
Assume that $\mu$ satisfies \eqref{eq:DD}.  Then
\begin{equation}\label{eq:Hmu-asympt}
  H_\mu(t)
  = \gamma_{n,s}\,A_\mu\,t^{-(n-s)/2}
    + o\bigl(t^{-(n-s)/2}\bigr)
  \qquad (t\to0^+),
\end{equation}
where
\begin{equation}\label{eq:gamma-ns}
  \gamma_{n,s}
  := (4\pi)^{-n/2}\,\gamma_s
  = (4\pi)^{-n/2}\,
    s\,2^{\,s-1}\Gamma\Bigl(\frac{s}{2}\Bigr)\,\vol(B^s).
\end{equation}
\end{lemma}

\begin{proof}
By Lemma \ref{lem:heat-spectral},
\[
  H_\mu(t)
  = \iint_{M\times M} p_t(x,y)\,d\mu(x)\,d\mu(y).
\]
Fix $r_0>0$ as in Lemma \ref{lem:heat-local-expansion} and split
\[
  H_\mu(t)
  = \iint_{d_g(x,y)\le r_0} p_t(x,y)\,d\mu(x)\,d\mu(y)
    + \iint_{d_g(x,y)>r_0} p_t(x,y)\,d\mu(x)\,d\mu(y)
  =: H_0(t)+H_1(t).
\]

By the Gaussian upper bound \eqref{eq:heat-off-diagonal},
\[
  0\le H_1(t)
  \le C\,t^{-n/2}\exp\Bigl(-\frac{r_0^2}{ct}\Bigr)\mu(M)^2
  = o\bigl(t^{-(n-s)/2}\bigr),
\]
since $t^{(n-s)/2}e^{-r_0^2/(ct)}\to0$ as $t\to0^+$.

For $H_0(t)$ we apply \eqref{eq:heat-local-expansion} with an
integer $N\ge1$ to be chosen momentarily.  For $d_g(x,y)\le r_0$ and
$0<t\le t_0$ we have
\[
  p_t(x,y)
  = (4\pi t)^{-n/2} e^{-d_g(x,y)^2/(4t)}
     \biggl(\sum_{k=0}^{N-1} t^k u_k(x,y)\biggr)
    + R_N(t,x,y),
\]
with $|R_N(t,x,y)|\le C_N t^{N-n/2}$.  Thus
\[
  H_0(t)
  = (4\pi t)^{-n/2}
      \sum_{k=0}^{N-1} t^k
      \iint_{d_g(x,y)\le r_0} e^{-d_g(x,y)^2/(4t)}u_k(x,y)\,d\mu(x)\,d\mu(y)
    + E_N(t),
\]
where
\[
  |E_N(t)|
  \le C_N t^{N-n/2}\mu(M)^2.
\]

We first single out the $k=0$ term.  Noting
\eqref{eq:u0-Lipschitz}, we may set $\omega:=u_0-1$, and write
\[
\begin{split}
  &\iint_{d_g(x,y)\le r_0} e^{-d_g(x,y)^2/(4t)}u_0(x,y)\,d\mu(x)\,d\mu(y) \\
  &= \iint_{d_g(x,y)\le r_0} e^{-d_g(x,y)^2/(4t)}\,d\mu(x)\,d\mu(y)
   + \iint_{d_g(x,y)\le r_0} e^{-d_g(x,y)^2/(4t)}\omega(x,y)\,d\mu(x)\,d\mu(y).
\end{split}
\]
The first integral differs from $G(t)$ by a tail integral:
\[
\begin{split}
  \iint_{d_g(x,y)\le r_0} e^{-d_g(x,y)^2/(4t)}\,d\mu(x)\,d\mu(y)
  &= G(t) - \iint_{d_g(x,y)>r_0} e^{-d_g(x,y)^2/(4t)}\,d\mu(x)\,d\mu(y) \\
  &= G(t) + T(t),
\end{split}
\]
where
\[
  |T(t)|
  \le \mu(M)^2 e^{-r_0^2/(4t)}
  = o(t^{s/2})
\]
as $t\to0^+$.

For the second integral we use \eqref{eq:u0-Lipschitz}:
\[
\begin{split}
  \biggl|\iint_{d_g(x,y)\le r_0}
          e^{-d_g(x,y)^2/(4t)}\omega(x,y)\,d\mu(x)\,d\mu(y)\biggr|
  &\le C_2
      \iint_{d_g(x,y)\le r_0} d_g(x,y)e^{-d_g(x,y)^2/(4t)}\,d\mu(x)\,d\mu(y) \\
  &= C_2\int_0^{r_0} r e^{-r^2/(4t)}\,dF(r) \\
  &=: C_2 J(t).
\end{split}
\]
Similar to the proof of Lemma \ref{lem:Gaussian-F}, noting $F(0)=0$ and \eqref{eq:DD}, integration by parts gives

\[
  |J(t)|
  \le C r_0^{s+1}e^{-r_0^2/(4t)}
  + C\int_0^{r_0} e^{-r^2/(4t)} r^s\,dr
  + \frac{C}{t}\int_0^{r_0} e^{-r^2/(4t)} r^{s+2}\,dr.
\]
Extending the integrals to $\infty$ and changing variables $r=\sqrt{t}\,u$ gives
\[
  \int_0^\infty e^{-r^2/(4t)} r^{a}\,dr
  = t^{(a+1)/2}\int_0^\infty e^{-u^2/4}u^{a}\,du
 \lesssim t^{(a+1)/2},
\]
so the last two terms are $O(t^{(s+1)/2})$ and $t^{-1}O(t^{(s+3)/2})=O(t^{(s+1)/2})$,
while the boundary term is $o(t^{(s+1)/2})$, and thus  $J(t)\lesssim t^{(s+1)/2}$.
Hence
\[
  \iint_{d_g(x,y)\le r_0}
          e^{-d_g(x,y)^2/(4t)}\omega(x,y)\,d\mu(x)\,d\mu(y)
  = O\bigl(t^{(s+1)/2}\bigr).
\]

Putting these estimates together, we obtain
\[
  \iint_{d_g(x,y)\le r_0} e^{-d_g(x,y)^2/(4t)}u_0(x,y)\,d\mu(x)\,d\mu(y)
  = G(t) + O\bigl(t^{(s+1)/2}\bigr).
\]
Therefore
\[
  (4\pi t)^{-n/2}
  \iint_{d_g(x,y)\le r_0} e^{-d_g(x,y)^2/(4t)}u_0(x,y)\,d\mu(x)\,d\mu(y)
  = (4\pi t)^{-n/2}G(t) + E_0(t),
\]
where
\[
  |E_0(t)|
  \lesssim t^{-n/2} t^{(s+1)/2}
  = t^{1/2-(n-s)/2}
  = o\bigl(t^{-(n-s)/2}\bigr)
\]
as $t\to0^+$.

Next, for each $k=1,\dots,N-1$, the smoothness of $u_k$ and the upper bound
\eqref{eq:G-upper} yield
\[
\begin{split}
  &\biggl|
   (4\pi t)^{-n/2}t^k
   \iint_{d_g(x,y)\le r_0} e^{-d_g(x,y)^2/(4t)}u_k(x,y)\,d\mu(x)\,d\mu(y)
   \biggr|\\
  &\qquad\le C_k t^{-n/2+k}
      \iint e^{-d_g(x,y)^2/(4t)}\,d\mu(x)\,d\mu(y)
   = C_k t^{-n/2+k}G(t)
   \lesssim t^{-n/2+k} t^{s/2} \\
  &\qquad= O\bigl(t^{k-(n-s)/2}\bigr)
   = o\bigl(t^{-(n-s)/2}\bigr),
\end{split}
\]
since $k\ge1$.

Finally, for the remainder term we have
\[
  |E_N(t)| \le C_N t^{N-n/2}\mu(M)^2.
\]
Choose an integer $N>\frac{s}{2}$.  Then
\[
  t^{(n-s)/2} |E_N(t)|
  \lesssim t^{N-s/2}\longrightarrow 0
  \qquad(t\to0^+),
\]
so $E_N(t)=o\bigl(t^{-(n-s)/2}\bigr)$.

Collecting the contributions from $k=0,\dots,N-1$, the remainder $E_N$, and
$H_1(t)$, we conclude that
\[
  H_\mu(t)
  = (4\pi t)^{-n/2}G(t) + o\bigl(t^{-(n-s)/2}\bigr)
  \qquad (t\to0^+).
\]
Using Lemma \ref{lem:Gaussian-F}, which gives
$G(t)=\gamma_s A_\mu t^{s/2}+o(t^{s/2})$, we obtain
\[
  H_\mu(t)
  = (4\pi)^{-n/2}\gamma_s A_\mu t^{-(n-s)/2}
     + o\bigl(t^{-(n-s)/2}\bigr),
\]
which is exactly \eqref{eq:Hmu-asympt} with $\gamma_{n,s}$ given by
\eqref{eq:gamma-ns}.
\end{proof}

\section{A Tauberian theorem and the proof of Theorem \texorpdfstring{\ref{thm:main}}{main}}\label{sec:Tauberian}

We now view the smoothed Kuznecov sum $H_\mu(t)$ as a Laplace--Stieltjes
transform of a spectral counting function and apply a Tauberian theorem to
deduce the asymptotic behavior of $N_\mu(\lambda)$.

\subsection{Laplace--Stieltjes representation}

Define a nondecreasing, right-continuous function $\mathcal N:[0,\infty)\to
[0,\infty)$ by
\[
  \mathcal N(S)
  := \sum_{\lambda_j^2\le S}
      \Bigl|\int_M e_j\,d\mu\Bigr|^2,
  \qquad S\ge0.
\]
Thus $\mathcal N$ puts mass $ |\int_M e_j\,d\mu|^2$ at the point $S=\lambda_j^2$.
By definition \eqref{eq:Kuz-sum},
\begin{equation}\label{eq:Sigma-vs-N}
  \mathcal N(\lambda^2) = N_\mu(\lambda)
  \qquad (\lambda\ge0).
\end{equation}

Using Lemma \ref{lem:heat-spectral} and the spectral expansion of $p_t$, we
have
\begin{equation}\label{eq:Hmu-Laplace}
  H_\mu(t)
  = \sum_{j=0}^\infty e^{-t\lambda_j^2}
       \Bigl|\int_M e_j\,d\mu\Bigr|^2
  = \int_{[0,\infty)} e^{-tS}\,d\mathcal N(S),
  \qquad t>0.
\end{equation}
In particular, $H_\mu(t)$ is the Laplace--Stieltjes transform of $\mathcal N$.

\subsection{A Tauberian theorem for Laplace--Stieltjes transforms}

We now recall a standard Tauberian theorem which connects the singular
behavior of a Laplace--Stieltjes transform at the origin to the growth of
the underlying measure.  This is a special case of Karamata's theorem. See,
for example, \cite[Theorem 1.7.1]{BinghamGoldieTeugels}.

\begin{theorem}[Karamata's Tauberian Theorem]\label{thm:Tauberian-heat}
Let $\mathcal N:[0,\infty)\to[0,\infty)$ be nondecreasing and right-continuous,
and let
\[
  \mathcal F(t) := \int_{[0,\infty)} e^{-tS}\,d\mathcal N(S)
\]
be finite for all $t>0$.  Suppose that for some constants $C>0$ and
$\beta>0$,
\[
  \mathcal F(t) \sim C\,t^{-\beta}
  \qquad (t\to0^+).
\]
Then
\[
  \mathcal N(S) \sim \frac{C}{\Gamma(\beta+1)}\,S^\beta
  \qquad (S\to\infty).
\]
\end{theorem}

Theorem \ref{thm:Tauberian-heat} is the Tauberian half of the usual
``Abelian $\Longleftrightarrow$ Tauberian'' correspondence for
Laplace--Stieltjes transforms with index $\beta>0$: roughly speaking,
asymptotics of $\mathcal N(S)$ as $S\to\infty$ and asymptotics of its
Laplace--Stieltjes transform $\mathcal F(t)$ as $t\to0^+$ determine each other.
A proof in this general framework can be found in
\cite[§1.7]{BinghamGoldieTeugels}.

\subsection{Proof of Theorem \ref{thm:main}}

We now combine Lemma \ref{lem:Hmu-asympt} with Theorem
\ref{thm:Tauberian-heat}.  Recall that Lemma \ref{lem:Hmu-asympt} shows that
\[
  H_\mu(t)
  = \gamma_{n,s}\,A_\mu\,t^{-(n-s)/2}
    + o\bigl(t^{-(n-s)/2}\bigr)
  \qquad (t\to0^+),
\]
with $\gamma_{n,s}$ given by \eqref{eq:gamma-ns}.  In the notation of
Theorem \ref{thm:Tauberian-heat} we have
\[
  \mathcal F(t) = H_\mu(t),\quad
  C = \gamma_{n,s}A_\mu,\quad
  \beta = \frac{n-s}{2}.
\]
Therefore
\[
  \mathcal N(S)
  \sim \frac{\gamma_{n,s}A_\mu}{\Gamma\bigl(\frac{n-s}{2}+1\bigr)}\,
        S^{(n-s)/2}
  \qquad (S\to\infty).
\]

Using \eqref{eq:Sigma-vs-N} with $S=\lambda^2$, we obtain
\[
  N_\mu(\lambda)
  = \mathcal N(\lambda^2)
  \sim \frac{\gamma_{n,s}A_\mu}{\Gamma\bigl(\tfrac{n-s}{2}+1\bigr)}\,
        \lambda^{n-s}
  \qquad (\lambda\to\infty).
\]
In particular, \eqref{eq:Kuz-main} holds with
\[
  C_{n,s}
  = \frac{\gamma_{n,s}}{\Gamma\bigl(\tfrac{n-s}{2}+1\bigr)}.
\]

It remains to simplify this constant.  Recall from \eqref{eq:gamma-ns} and
\eqref{eq:gamma-s} that
\[
  \gamma_{n,s}
  = (4\pi)^{-n/2}\,
    s\,2^{\,s-1}\Gamma\Bigl(\frac{s}{2}\Bigr)\,\vol(B^s),
\]
and that
\[
  \vol(B^m)
  = \frac{\pi^{m/2}}{\Gamma\bigl(\frac{m}{2}+1\bigr)},
  \qquad m>0.
\]
Since $\Gamma\bigl(\frac{s}{2}+1\bigr)=\frac{s}{2}\Gamma\bigl(\frac{s}{2}\bigr)$,
we may write
\[
  \vol(B^s)
  = \frac{\pi^{s/2}}{\Gamma\bigl(\frac{s}{2}+1\bigr)}
  = \frac{2}{s}\,\pi^{s/2}\,\Gamma\Bigl(\frac{s}{2}\Bigr)^{-1}.
\]
Substituting this into the formula for $\gamma_{n,s}$ gives
\[
  \gamma_{n,s}
  = (4\pi)^{-n/2}\,s\,2^{\,s-1}\Gamma\Bigl(\frac{s}{2}\Bigr)
     \cdot \frac{2}{s}\,\pi^{s/2}\,\Gamma\Bigl(\frac{s}{2}\Bigr)^{-1}
  = (4\pi)^{-n/2}\,2^{\,s}\,\pi^{s/2}.
\]
Hence
\[
  C_{n,s}
  = \frac{(4\pi)^{-n/2}2^{\,s}\pi^{s/2}}
         {\Gamma\bigl(\frac{n-s}{2}+1\bigr)}
  = 2^{-(n-s)}\pi^{-(n-s)/2}
      \frac{1}{\Gamma\bigl(\frac{n-s}{2}+1\bigr)}.
\]
On the other hand, taking $m=n-s$ in the volume formula,
\[
  \vol(B^{n-s})
  = \frac{\pi^{(n-s)/2}}{\Gamma\bigl(\frac{n-s}{2}+1\bigr)},
\]
so
\[
  (2\pi)^{-(n-s)}\,\vol(B^{n-s})
  = 2^{-(n-s)}\pi^{-(n-s)}\cdot
    \frac{\pi^{(n-s)/2}}{\Gamma\bigl(\frac{n-s}{2}+1\bigr)}
  = C_{n,s}.
\]
Thus
\[
  C_{n,s}
  = (2\pi)^{-(n-s)}\,\vol(B^{\,n-s}),
\]
which is exactly the constant stated in Theorem \ref{thm:main}.  This
completes the proof.

\section{Formulae of Hare--Roginskaya and the proof of Theorem \ref{thm:HR-bound}}\label{sec:HR}

In this section we explain how the energy identities of Hare and Roginskaya
\cite{HareRoginskaya03} imply the polynomial upper bounds
stated in Theorem \ref{thm:HR-bound}. 

Recall that for $0<u<n$ we defined the Riesz energy of a finite Borel measure
$\mu$ on $M$ by
\[
  I_u(\mu) = \iint d_g(x,y)^{-u}\,d\mu(x)\,d\mu(y),
\]
where $d_g(x,y)$ denotes the Riemannian distance on $(M,g)$, see
\eqref{eq:Riesz-energy}.  Hare and Roginskaya relate $I_u(\mu)$ to a
weighted spectral sum via the kernel of a suitable elliptic operator.  For
our purposes, we only need the following consequence of their results.

\begin{proposition}[\cite{HareRoginskaya03}]\label{prop:HR-energy}
Let $(M,g)$ be a compact Riemannian manifold of dimension $n\ge2$.  For each
$u\in(0,n)$ there exists a constant $c_{n,u}\ge1$, depending only on $n$, $u$
and $(M,g)$, such that for every finite Borel measure $\mu$ on $M$,
\begin{equation}\label{eq:HR-energy}
  c_{n,u}^{-1} I_u(\mu)
  \le \sum_{j=0}^\infty (1+\lambda_j^2)^{-(n-u)/2}
        \Bigl|\int_M e_j\,d\mu\Bigr|^2
  \le c_{n,u}\, I_u(\mu).
\end{equation}
\end{proposition}



Given Proposition \ref{prop:HR-energy}, the proof of
Theorem \ref{thm:HR-bound} is straightforward.

\subsection*{Proof of Theorem \ref{thm:HR-bound}}

Let $u\in(0,n)$ and let $\mu$ be a finite Borel measure with
$I_u(\mu)<\infty$.  By the upper bound in \eqref{eq:HR-energy},
\begin{equation}\label{eq:weighted-sum-bound}
  \sum_{j=0}^\infty (1+\lambda_j^2)^{-(n-u)/2}
     \Bigl|\int_M e_j\,d\mu\Bigr|^2
  \le c_{n,u}\, I_u(\mu).
\end{equation}

Now fix $\lambda\ge1$.  Since the weight
$(1+\lambda_j^2)^{-(n-u)/2}$ is decreasing in $\lambda_j$, we have
\[
  (1+\lambda_j^2)^{-(n-u)/2}
  \ge (1+\lambda^2)^{-(n-u)/2}
  \qquad\text{whenever }\lambda_j\le\lambda,
\]
and hence
\[
  \sum_{\lambda_j\le\lambda} \Bigl|\int_M e_j\,d\mu\Bigr|^2
  \le (1+\lambda^2)^{\frac{n-u}{2}}
      \sum_{j=0}^\infty (1+\lambda_j^2)^{-(n-u)/2}
        \Bigl|\int_M e_j\,d\mu\Bigr|^2.
\]
Combining this with \eqref{eq:weighted-sum-bound} gives
\[
  \sum_{\lambda_j\le\lambda} \Bigl|\int_M e_j\,d\mu\Bigr|^2
  \le C_1(M,g,u,\mu)\,(1+\lambda^2)^{\frac{n-u}{2}}
  \le C_1(M,g,u,\mu)\,\lambda^{\,n-u}
\]
for all $\lambda\ge1$, where we may take
\[
  C_1(M,g,u,\mu)
  := 2^{\frac{n-u}{2}}c_{n,u}\,I_u(\mu).
\]
This proves the first assertion of Theorem \ref{thm:HR-bound}.

To obtain the second statement of Theorem \ref{thm:HR-bound}, it remains to note
that if $\mu$ has correlation dimension $\dim_C(\mu)=s$ for some $s\in(0,n)$,
then $I_u(\mu)<\infty$ for every $u<s$. This is a standard estimate, and is the
reason why the correlation dimension is sometimes called the energy dimension.
We include a brief argument in our setting for completeness.

\begin{lemma}\label{lem:finite-energy}
Let $\mu$ be a finite Borel measure on $M$ with correlation dimension
$\dim_C(\mu)=s$ for some $s\in(0,n)$, and let $u\in(0,s)$. Then $I_u(\mu)<\infty$.
\end{lemma}

\begin{proof}
Recall that
\[
  F(r)
  = (\mu\times\mu)\bigl(\{(x,y): d_g(x,y)\le r\}\bigr)
  = \int_M \mu(B(x,r))\,d\mu(x).
\]
Using the layer-cake representation, we have
\[
  I_u(\mu)
  = \int_0^{\diam(M)} u r^{-u-1} F(r)\,dr + \mu(M)^2\,\diam(M)^{-u}.
\]

Fix $\varepsilon>0$ so that $u<s-\varepsilon$. Since $\dim_C(\mu)=s$, there exist
constants $C>0$ and $r_0>0$ such that
\[
  F(r)\le C r^{s-\varepsilon}\qquad(0<r\le r_0).
\]
Split the integral at $r_0$:
\[
  I_u(\mu)
  = u\int_0^{r_0} r^{-u-1}F(r)\,dr
    + u\int_{r_0}^{\diam(M)} r^{-u-1}F(r)\,dr
    + \mu(M)^2\,\diam(M)^{-u}.
\]
On $(0,r_0]$,
\[
  u\int_0^{r_0} r^{-u-1}F(r)\,dr
  \le uC\int_0^{r_0} r^{s-\varepsilon-u-1}\,dr
  <\infty,
\]
since $s-\varepsilon-u>0$. On $[r_0,\diam(M)]$,
\[
  u\int_{r_0}^{\diam(M)} r^{-u-1}F(r)\,dr
  \le u\mu(M)^2\int_{r_0}^{\diam(M)} r^{-u-1}\,dr
  <\infty.
\]
Thus $I_u(\mu)<\infty$ for every $u<s$, as claimed.
\end{proof}

Applying Lemma \ref{lem:finite-energy} with any fixed $u<s$ shows
that the first part of Theorem \ref{thm:HR-bound} applies to $\mu$ and yields
a constant $C_2=C_2(M,g,s,u,\mu)$ such that
\[
  \sum_{\lambda_j\le\lambda} \Bigl|\int_M e_j\,d\mu\Bigr|^2
  \le C_2\,\lambda^{\,n-u}
  \qquad (\lambda\ge1),
\]
without any density assumption.  This completes the proof of
Theorem \ref{thm:HR-bound}.

\begin{remark}
Theorem \ref{thm:HR-bound} shows that for each fixed $u<s$ one can control the
growth of the Kuznecov sum by $\lambda^{n-u}$ under the mild hypothesis
$I_u(\mu)<\infty$.  Our
main theorem, Theorem \ref{thm:main}, goes further by identifying the sharp
exponent $n-s$ and the exact leading constant in the asymptotic
\eqref{eq:Kuz-main}, under the additional assumption of the 
averaged $s$-density condition.  As discussed in the introduction, the methods
of \cite{HareRoginskaya03} treat $u$ as a fixed parameter
and do not provide the uniform control in $u$ near $s$ that is needed to
pass to the endpoint $u=s$. This is why we turn to a more refined analysis
based on the heat kernel and Tauberian theory.
\end{remark}

\section{Sharpness and examples}\label{sec:sharpness}

In this section, we discuss the sharpness of Theorem \ref{thm:main} and provide a few examples. We explain
the necessity of the averaged $s$-density condition \eqref{eq:DD} for obtaining a
genuine one-term asymptotic, and we show that the remainder $o(\lambda^{n-s})$ in
\eqref{eq:Kuz-main} is essentially optimal at this level of generality. In
particular, it cannot be improved uniformly to a power-saving error term, even
over natural regular subclasses of measures. In the final part of this section we present an explicit self-similar fractal measure
whose distance distribution $F(r)$ admits an averaged $s$-density constant, and hence
satisfies \eqref{eq:DD}.

\subsection{Necessity of the averaged $s$-density condition}\label{subsec:necessity-density}

In this subsection we prove Proposition \ref{prop:need-DD}, which shows that the
averaged $s$-density hypothesis in Theorem \ref{thm:main} is not merely a
technical convenience. Without a stabilization of the small-scale behavior of
$\mu$, one cannot in general expect the normalized quantities
$N_\mu(\lambda)/\lambda^{n-s}$ to converge.

Recall the distance distribution
\[
  F(r):=(\mu\times\mu)\bigl(\{(x,y):d_g(x,y)\le r\}\bigr),\qquad r\ge0,
\]
and the Gaussian average
\[
  G(t):=\iint e^{-d_g(x,y)^2/(4t)}\,d\mu(x)\,d\mu(y)
       =\int_0^{\diam(M)} e^{-r^2/(4t)}\,dF(r).
\]
Define the normalized coefficient
\[
  a(r):=\frac{F(r)}{r^s}\qquad(r>0).
\]
If $\mu$ is $s$-Ahlfors--David regular then $a(r)$ is bounded above and below by
positive constants for $0<r\le r_0$.

Since $F(0)=0$, an integration by parts yields, up to a boundary term that decays exponentially in $1/t$,
\begin{equation}\label{eq:G-average-a}
  G(t)
  = \frac{1}{2t}\int_0^{\diam(M)} r e^{-r^2/(4t)}\,F(r)\,dr.
\end{equation}
Changing variables $r=\sqrt{t}\,u$ yields
\begin{equation}\label{eq:G-average-a2}
  \frac{G(t)}{t^{s/2}}
  = \int_0^{\diam(M)/\sqrt t} w_s(u)\,a(\sqrt t\,u)\,du,
  \qquad
  w_s(u):=\frac12\,u^{s+1}e^{-u^2/4}.
\end{equation}
The weight $w_s$ is positive and integrable on $(0,\infty)$, with
\[
  \kappa_s:=\int_0^\infty w_s(u)\,du
  = \frac12\int_0^\infty u^{s+1}e^{-u^2/4}\,du
  = s\,2^{s-1}\Gamma\Bigl(\frac{s}{2}\Bigr).
\]

Let
\[
  a_-:=\liminf_{r\to0^+} a(r),\qquad a_+:=\limsup_{r\to0^+} a(r).
\]
Using \eqref{eq:G-average-a2}, the boundedness of $a$, and $\int_0^\infty w_s<\infty$,
one obtains the general bounds
\begin{equation}\label{eq:G-liminf-limsup-general}
  \kappa_s\,a_-
  \le \liminf_{t\to0^+}\frac{G(t)}{t^{s/2}}
  \le \limsup_{t\to0^+}\frac{G(t)}{t^{s/2}}
  \le \kappa_s\,a_+.
\end{equation}
In general \eqref{eq:G-liminf-limsup-general} may be strict: the Gaussian
average is a \emph{weighted} average of the values of $a(r)$ on the scale
$r\approx\sqrt t$, so thin oscillations in $a$ can be smoothed.

However, in the model constructions relevant for sharpness one can arrange that
different limit points of $a(r)$ persist across long blocks of
scales.  In that setting the endpoints in \eqref{eq:G-liminf-limsup-general} are
realized along suitable subsequences.

\begin{lemma}[Realizing $\liminf/\limsup$]\label{lem:realize-liminf-limsup}
Assume $a$ is bounded on $(0,\diam (M)]$, and let
\[
  a_-:=\liminf_{r\to0^+} a(r),\qquad a_+:=\limsup_{r\to0^+} a(r).
\]
\begin{enumerate}
\item If there exist $r_m\to0^+$ and $U_m\to\infty$ such that
\[
  \sup_{\,r_m/U_m \le r \le U_m r_m}\,|a(r)-a_-|\longrightarrow 0,
\]
then with $t_m:=r_m^2$ one has
\[
  \frac{G(t_m)}{t_m^{s/2}}\longrightarrow \kappa_s\,a_-.
\]
\item If there exist $\rho_m\to0^+$ and $V_m\to\infty$ such that
\[
  \sup_{\,\rho_m/V_m \le r \le V_m \rho_m}\,|a(r)-a_+|\longrightarrow 0,
\]
then with $\tau_m:=\rho_m^2$ one has
\[
  \frac{G(\tau_m)}{\tau_m^{s/2}}\longrightarrow \kappa_s\,a_+.
\]
\end{enumerate}
\end{lemma}

\begin{proof}
We prove (1); (2) is identical.  Put $t_m=r_m^2$ and use \eqref{eq:G-average-a2}:
\[
  \frac{G(t_m)}{t_m^{s/2}}
  = \int_0^{\diam(M)/r_m} w_s(u)\,a(r_m u)\,du.
\]
Split the integral into three pieces:
\[
  \int_0^{\diam(M)/r_m}
  = \int_0^{1/U_m}+\int_{1/U_m}^{U_m}+\int_{U_m}^{\diam(M)/r_m}.
\]
On $u\in[1/U_m,U_m]$ we have $r_m u\in[r_m/U_m,\,U_m r_m]$, hence
$a(r_m u)\to a_-$ uniformly, so
\[
  \int_{1/U_m}^{U_m} w_s(u)\,a(r_m u)\,du
  = a_-\int_{1/U_m}^{U_m} w_s(u)\,du + o(1).
\]
Since $a$ is bounded, say $|a|\le K$ on $(0,r_0]$, the tails satisfy
\[
  \left|\int_0^{1/U_m} w_s(u)\,a(r_m u)\,du\right|
  \le K\int_0^{1/U_m} w_s(u)\,du \longrightarrow 0,
\]
because $w_s(u)\approx u^{s+1}$ near $0$, and
\[
  \left|\int_{U_m}^{\diam(M)/r_m} w_s(u)\,a(r_m u)\,du\right|
  \le K\int_{U_m}^{\infty} w_s(u)\,du \longrightarrow 0,
\]
because $w_s\in L^1(0,\infty)$ and $U_m\to\infty$.  Finally,
\[
  \int_{1/U_m}^{U_m} w_s(u)\,du \longrightarrow \int_0^\infty w_s(u)\,du=\kappa_s,
\]
so $G(t_m)/t_m^{s/2}\to \kappa_s a_-$.
\end{proof}

Lemma \ref{lem:realize-liminf-limsup} shows that if $a(r)$ has two distinct limit
points $a_-<a_+$ and one can arrange two sequences of logarithmic blocks on which
$a(r)$ stays uniformly close to $a_-$ and $a_+$, respectively, then
$G(t)/t^{s/2}$ has two distinct subsequential limits $\kappa_s a_-$ and $\kappa_s a_+$,
and hence does not converge.  By the near-diagonal reduction in
Section \ref{sec:heat}, this forces $t^{(n-s)/2}H_\mu(t)$, and therefore
$N_\mu(\lambda)/\lambda^{n-s}$, to fail to converge as well.  This explains why
some stabilization of $F(r)/r^s$ is needed for a
genuine one-term asymptotic.

\subsection{Optimality of the remainder}\label{subsec:optimal-remainder}

In the smooth submanifold case, Zelditch's Kuznecov formula \eqref{eq:Zelditch}
comes with an $O(\lambda^{n-k-1})$ remainder term. Moreover, for generic metrics
this remainder can be improved to $o(\lambda^{n-k-1})$ by a recent result of
Kaloshin, Wyman and the author \cite{kaloshin2025kuznecov}. In contrast, the
asymptotic formula in Theorem \ref{thm:main} is stated with an
$o(\lambda^{n-s})$ remainder.

It is natural to ask whether one can strengthen \eqref{eq:Kuz-main} to a
\emph{uniform} power-saving error term of the form
\begin{equation}\label{eq:uniform-powersaving}
  N_\mu(\lambda)=C_{n,s}A_\mu\,\lambda^{n-s}+O(\lambda^{n-s-\delta})
  \qquad(\lambda\to\infty),
\end{equation}
for some $\delta>0$ depending only on $(M,g,n,s)$ and valid for all measures $\mu$
satisfying the hypotheses of Theorem \ref{thm:main}. Proposition \ref{thm:no-uniform-ps} asserts that such a uniform improvement is
not available in general, even under additional regularity assumptions, and we
prove it in this subsection.

\medskip

\noindent\textbf{Step 1: from $N_\mu$ to the Laplace--Stieltjes transform $H_\mu$.}
Write $\mathcal N(S):=N_\mu(\sqrt S)$ as in \eqref{eq:Sigma-vs-N}, so that
\[
  H_\mu(t)=\int_{[0,\infty)} e^{-tS}\,d\mathcal N(S).
\]
If \eqref{eq:uniform-powersaving} holds for some $\delta>0$, then
\[
  \mathcal N(S)=C_{n,s}A_\mu\,S^{\frac{n-s}{2}}+O\Bigl(S^{\frac{n-s-\delta}{2}}\Bigr),
\]
and a standard Abelian (integration-by-parts) computation yields
\begin{equation}\label{eq:Abelian-remainder}
  H_\mu(t)
  = C_{n,s}A_\mu\,\Gamma\Bigl(\frac{n-s}{2}+1\Bigr)\,t^{-\frac{n-s}{2}}
    + O\Bigl(t^{-\frac{n-s-\delta}{2}}\Bigr)
  \qquad(t\to0^+).
\end{equation}

\medskip

\noindent\textbf{Step 2: from $H_\mu$ to the Gaussian average $G(t)$.}
Recall
\[
  G(t):=\iint e^{-d_g(x,y)^2/(4t)}\,d\mu(x)\,d\mu(y).
\]
Inspecting the proof of Lemma \ref{lem:Hmu-asympt} (using the Lipschitz bound
\eqref{eq:u0-Lipschitz}, the estimate $J(t)\lesssim t^{(s+1)/2}$, and choosing an
integer $N\ge \lceil (s+1)/2\rceil$ in the local heat-kernel parametrix) yields the
quantitative refinement
\begin{equation}\label{eq:H-vs-G-rate}
  H_\mu(t)=(4\pi t)^{-n/2}G(t) + O\left(t^{-\frac{n-s}{2}+\frac12}\right)
  \qquad(t\to0^+).
\end{equation}
Combining \eqref{eq:Abelian-remainder} and \eqref{eq:H-vs-G-rate}, and using that
$C_{n,s}=\gamma_{n,s}/\Gamma(\frac{n-s}{2}+1)$ and $\gamma_{n,s}=(4\pi)^{-n/2}\gamma_s$,
one obtains (for $\delta\in(0,1)$)
\begin{equation}\label{eq:G-power-saving}
  \frac{G(t)}{t^{s/2}}
  = \gamma_s A_\mu + O(t^{\delta/2})
  \qquad(t\to0^+),
\end{equation}
where $\gamma_s$ is as in Lemma \ref{lem:Gaussian-F}.  Thus, a uniform power saving
in \eqref{eq:uniform-powersaving} would force a uniform power saving in $t$ for
$G(t)/t^{s/2}$.

\medskip

\noindent\textbf{Step 3: from $G(t)$ to $F(r)$.}
Recall $F(r):=(\mu\times\mu)\{(x,y):d_g(x,y)\le r\}$ and set
\[
  q(r):=\frac{F(r)}{\vol(B^s)\,A_\mu\,r^s}.
\]
Then by definition, $q(r)\to1$
as $r\to0^+$.  Integrating by parts in
$G(t)=\int_0^{\diam(M)} e^{-r^2/(4t)}\,dF(r)$ and changing variables $r=\sqrt t\,u$
yields
\begin{equation}\label{eq:G-average-q}
  \frac{G(t)}{\vol(B^s)\,A_\mu\,t^{s/2}}
  = \int_0^{\diam(M)/\sqrt t} w_s(u)\,q(\sqrt t\,u)\,du + O(e^{-c/t}),
\end{equation}
where
\[
  w_s(u):=\frac12\,u^{s+1}e^{-u^2/4}
\]
is strictly positive and integrable on $(0,\infty)$.

\begin{lemma}\label{lem:block-deviation-G}
Assume $q(r)\to1$ as $r\to0^+$ and $\|q\|_{L^\infty(0,\diam(M)]}\le C_0$ for some $C_0\ge1$.
Fix $\eta\in(0,1)$ and choose $U=U(\eta)\ge2$ so that
\begin{equation}\label{eq:choose-U}
  \int_U^\infty w_s(u)\,du \le \frac{\eta}{8C_0}\int_1^U w_s(u)\,du.
\end{equation}
Fix $r>0$, put $t=r^2$, and assume that $\text{for all }\rho\in[r,Ur].$
\[
  q(\rho)\le 1-\eta.
\]
Then for all sufficiently small $r$ one has
\[
  \left|\frac{G(t)}{t^{s/2}}-\gamma_sA_\mu\right|
  \ge c_s\,A_\mu\,\eta\int_1^U w_s(u)\,du,
\]
where $c_s>0$ depends only on $s$.
\end{lemma}

\begin{proof}
Let $\kappa_s:=\int_0^\infty w_s(u)\,du$, so $\gamma_s=\vol(B^s)\kappa_s$.  Using
\eqref{eq:G-average-q} with $t=r^2$ and absorbing the exponentially small term, we get
\[
  \frac{G(r^2)}{\vol(B^s)\,A_\mu\,r^s}-\kappa_s
  = \int_0^\infty w_s(u)\bigl(q(ru)-1\bigr)\,du + o(1)
  \qquad (r\to0^+).
\]
Split the integral into $u\in(0,1)$, $u\in[1,U]$, and $u\in[U,\infty)$.

On $[1,U]$ we have $ru\in[r,Ur]$, so $q(ru)-1$ has fixed sign and $|q(ru)-1|\ge\eta$.
Hence
\[
  \left|\int_1^U w_s(u)\bigl(q(ru)-1\bigr)\,du\right|
  \ge \eta\int_1^U w_s(u)\,du.
\]
On $[U,\infty)$ we use $|q-1|\le C_0+1\le 2C_0$ and \eqref{eq:choose-U}:
\[
  \left|\int_U^\infty w_s(u)\bigl(q(ru)-1\bigr)\,du\right|
  \le 2C_0\int_U^\infty w_s(u)\,du
  \le \frac{\eta}{4}\int_1^U w_s(u)\,du.
\]
Finally, since $q(\rho)\to1$ as $\rho\to0^+$ and $w_s\in L^1(0,1)$, dominated convergence gives
\[
  \int_0^1 w_s(u)\bigl(q(ru)-1\bigr)\,du \longrightarrow 0
  \qquad (r\to0^+).
\]
Therefore for all sufficiently small $r$,
\[
  \left|\int_0^\infty w_s(u)\bigl(q(ru)-1\bigr)\,du\right|
  \ge \frac{\eta}{2}\int_1^U w_s(u)\,du.
\]
Multiplying by $\vol(B^s)A_\mu$ and using $\gamma_s=\vol(B^s)\kappa_s$ yields the claim.
\end{proof}

\medskip

\noindent\textbf{Step 4: constructing arbitrarily slow convergence.}
We now construct an explicit example (for a fixed $\delta>0$) for which no
uniform remainder $O(\lambda^{n-s-\delta})$ can hold.  For simplicity, we take $s$ to be an integer, and
\[
  M=\mathbb T^n,\qquad H=\mathbb T^s\subset\mathbb T^n
\]
with the flat metric, so that $dV_H$ is Lebesgue measure on $\mathbb T^s$ and,
for all $\rho\in(0,1/10]$, the intrinsic ball $B_H(x,\rho)$ is a translate of
the Euclidean ball. We remark that similar constructions can be made for a general manifold--submanifold pair $(M,H)$.
 Let
\[
  K_\rho(y):=\mathbf 1_{B_H(0,\rho)}(y),\qquad
  \mathcal A_\rho f(x):=\frac{1}{\vol(B^s)\rho^s}\int_{\mathbb T^s} f(x-y)K_\rho(y)\,dV_H(y).
\]
Then whenever $d\mu=w\,dV_H$ we have
\[
  \frac{F(\rho)}{\vol(B^s)\rho^s}
  =\int_H w(x)\,\mathcal A_\rho w(x)\,dV_H(x).
\]

Choose frequencies $k_m\in\mathbb Z^s\setminus\{0\}$ such that $|k_m|$ is strictly
increasing and $k_m\neq \pm k_{m'}$ for $m\neq m'$, and define
\[
  \phi_m(x):=\cos(2\pi k_m\cdot x),\qquad
  \|\phi_m\|_{L^\infty(H)}\le 1,\qquad
  \int_H \phi_m\,dV_H=0,\qquad
  \int_H \phi_m^2\,dV_H=\tfrac12\,\vol(H).
\]
For $\rho\in(0,1/10]$ the averaging convolution operator $\mathcal A_\rho$ acts diagonally on
Fourier modes: there exists a continuous function $m_s:[0,\infty)\to\mathbb R$
with
\[
  m_s(0)=1,\qquad \lim_{\tau\to\infty} m_s(\tau)=0,\qquad |m_s(\tau)|\le 1,
\]
such that
\[
  \mathcal A_\rho \phi_m = m_s(2\pi |k_m|\rho)\,\phi_m .
\]

Fix amplitudes $a_m>0$ such that
\begin{equation}\label{eq:am-fourier}
  \sum_{m\ge1} a_m\le \tfrac14,
  \qquad\text{and hence}\qquad
  \sum_{m\ge1} a_m^2<\infty.
\end{equation}
Define
\[
  w(x):=1+\sum_{m\ge1} a_m\,\phi_m(x),
  \qquad\text{and}\qquad
  d\mu:=w\,dV_H.
\]
Then $\tfrac34\le w\le \tfrac54$, so $\mu$ is $s$-Ahlfors--David regular. Moreover, since
$\mu$ is absolutely continuous with respect to $dV_H$ with bounded density $w$,
$\mu$ admits a pointwise $s$-density $\theta=w$.  Hence, by Remark \ref{rem:density-implies-DD},
\[
  A_\mu=\int_H \theta\,d\mu=\int_H w^2\,dV_H.
\]

Using orthogonality of distinct Fourier modes,
\begin{equation}\label{eq:A-mu-fourier}
  A_\mu
  =\int_H w^2\,dV_H
  =\vol(H)+\frac12\,\vol(H)\sum_{m\ge1} a_m^2 .
\end{equation}
Moreover, for $\rho\in(0,1/10]$,
\begin{align}\label{eq:F-fourier}
  \frac{F(\rho)}{\vol(B^s)\rho^s}
  &=\int_H w\,\mathcal A_\rho w\,dV_H \notag\\
  &=\vol(H)
    +\frac12\,\vol(H)\sum_{m\ge1} a_m^2\,m_s(2\pi |k_m|\rho).
\end{align}
Therefore
\begin{equation}\label{eq:q-fourier}
  q(\rho)
  =\frac{F(\rho)}{\vol(B^s)A_\mu\,\rho^s}
  =\frac{1+\frac12\sum_{m\ge1} a_m^2\,m_s(2\pi |k_m|\rho)}
         {1+\frac12\sum_{m\ge1} a_m^2}.
\end{equation}
Since $\sum a_m^2<\infty$ and $|m_s|\le 1$, dominated convergence and $m_s(0)=1$
give $q(\rho)\to 1$ as $\rho\to0^+$.

Choose numbers $0<\varepsilon_0<1<L_0$ such that
\begin{equation}\label{eq:m-thresholds}
  |m_s(\tau)-1|\le \tfrac1{100}\quad\text{for }0\le \tau\le \varepsilon_0,
  \qquad
  |m_s(\tau)|\le \tfrac1{100}\quad\text{for }\tau\ge L_0.
\end{equation}
(This is possible by continuity at $0$ and $\lim_{\tau\to\infty}m_s(\tau)=0$.)

For each $m$ set
\[
  \eta_m:=\frac{1}{100}\,a_m^2,
\]
and let $U_m:=U(\eta_m)\ge2$ be the scale parameter given by
Lemma \ref{lem:block-deviation-G}.  Define
\[
  r_m:=\frac{L_0}{2\pi |k_m|}\,.
\]
We choose the frequencies inductively so that $|k_m|$ is strictly increasing and
\begin{equation}\label{eq:freq-sep}
  2\pi |k_{m-1}|\,U_m r_m \le \varepsilon_0,
  \qquad
  U_m r_m \le \tfrac1{10},
  \qquad
  r_m^\delta \le \frac{1}{10m}\,\eta_m .
\end{equation}
This is possible because once $a_m$ (hence $\eta_m$ and $U_m$) is fixed, we are
free to take $|k_m|$ arbitrarily large, which makes $r_m$ arbitrarily small.

Fix $m$ and take any $\rho\in[r_m,U_m r_m]$.  Then
\[
  2\pi |k_m|\rho \ge 2\pi |k_m|r_m = L_0,
\]
so \eqref{eq:m-thresholds} gives $|m_s(2\pi |k_m|\rho)|\le 1/100$, and the same
bound holds for all $j\ge m$ since $|k_j|\ge |k_m|$.  On the other hand,
\eqref{eq:freq-sep} implies, for $j<m$,
\[
  2\pi |k_j|\rho \le 2\pi |k_{m-1}|\,U_m r_m \le \varepsilon_0,
\]
so \eqref{eq:m-thresholds} gives $m_s(2\pi |k_j|\rho)\ge 1-1/100$ for all $j<m$.

Let $A:=\sum_{j<m}a_j^2$ and $T:=\sum_{j\ge m}a_j^2$ (so $T\ge a_m^2$).  Using the
bounds above in \eqref{eq:q-fourier} yields, for $\rho\in[r_m,U_m r_m]$,
\[
  q(\rho)
  \le
  \frac{1+\frac12\bigl(A+\tfrac1{100}T\bigr)}
       {1+\frac12(A+T)}
  \le
  1-\frac{1-\tfrac1{100}}{1+\frac12(A+T)}\cdot \frac{T}{2}.
\]
Since $\sum_{m\ge1}a_m^2\le \sum_{m\ge1}a_m\le \tfrac14$, we have $A+T\le \tfrac14$
and hence $1+\frac12(A+T)\le \tfrac98$.  Therefore
\[
  q(\rho)\le 1-\frac{4}{9}\Bigl(1-\tfrac1{100}\Bigr)T
  \le 1-\eta_m,
\]
so $q(\rho)\le 1-\eta_m$ on the entire block $[r_m,U_m r_m]$.

Applying Lemma \ref{lem:block-deviation-G} with $r=r_m$ and $\eta=\eta_m$ gives
\[
  \left|\frac{G(r_m^2)}{(r_m^2)^{s/2}}-\gamma_sA_\mu\right|
  =\left|\frac{G(r_m^2)}{r_m^{s}}-\gamma_sA_\mu\right|
  \ge c_s\,A_\mu\,\eta_m .
\]
Finally, the last condition in \eqref{eq:freq-sep} implies
\[
  \frac{\eta_m}{(r_m^2)^{\delta/2}}
  =\frac{\eta_m}{r_m^\delta}\ge 10m\longrightarrow\infty,
\]
so no estimate of the form
$\bigl|G(t)/t^{s/2}-\gamma_sA_\mu\bigr|=O(t^{\delta/2})$ can hold as $t\to0^+$.
By the implication \eqref{eq:G-power-saving}, this contradicts the consequence of a
uniform remainder $O(\lambda^{n-s-\delta})$.  Therefore no uniform power saving
$O(\lambda^{n-s-\delta})$ can hold over our class, and the $o(\lambda^{n-s})$
remainder in Theorem \ref{thm:main} is essentially sharp.

\subsection{A self-similar example on the unit interval}\label{subsec:example-nonlattice}

We record an explicit self-similar fractal measure on $[0,1]$ of non-integer dimension
satisfying the averaged $s$-density condition \eqref{eq:DD}.

\begin{proposition}\label{prop:nonlattice-example}
Let $\phi_0,\phi_1:[0,1]\to[0,1]$ be
\[
  \phi_0(x)=\tfrac12 x,\qquad \phi_1(x)=\tfrac23+\tfrac13 x,
\]
and let $K\subset[0,1]$ be the attractor $K=\phi_0(K)\cup\phi_1(K)$.
Let $s\in(0,1)$ be the unique solution of $2^{-s}+3^{-s}=1$, and set
$p_0:=2^{-s}$, $p_1:=3^{-s}$.
Let $\mu$ be the unique Borel probability measure supported on $K$ satisfying
\begin{equation}\label{eq:selfsim-mu}
  \mu(A)=p_0\,\mu(\phi_0^{-1}(A))+p_1\,\mu(\phi_1^{-1}(A))
  \qquad\text{for all Borel }A\subset[0,1].
\end{equation}
Then $s\notin\mathbb N$, and $\mu$ admits an averaged $s$-density constant $A_\mu\in(0,\infty)$ in the
sense of \eqref{eq:DD}. 
\end{proposition}

\begin{proof}
It is clear that $s$ is not an integer. For $0<r<1/6$, pairs from different first-level pieces cannot satisfy $|x-y|\le r$, hence expanding
$\mu\times\mu$ using \eqref{eq:selfsim-mu} gives the exact recursion
\[
  F(r)=p_0^2\,F(2r)+p_1^2\,F(3r)\qquad(0<r<1/6).
\]
Setting $f(r):=F(r)/(\vol(B^s)\,r^s)$ and using $p_0=2^{-s}$, $p_1=3^{-s}$ yields
\[
  f(r)=p_0\,f(2r)+p_1\,f(3r)\qquad(0<r<1/6).
\]
With $t=-\log r$ and $g(t):=f(e^{-t})$, we obtain
\[
  g(t)=p_0\,g(t-\log 2)+p_1\,g(t-\log 3)\qquad(t>\log 6).
\]
Since $\log 2/\log 3\notin\mathbb Q$, the key renewal theorem (see, e.g. \cite{Feller1971Vol2}) implies that
$g(t)\to A_\mu$ as $t\to\infty$ for some $A_\mu\in(0,\infty)$, i.e. $f(r)\to A_\mu$ as $r\to0^+$,
which is exactly \eqref{eq:DD} with exponent $s$.
Finally, we remark that this measure $\mu$ in our example satisfies Ahlfors--David regularity.
\end{proof}

\bibliographystyle{alpha}
\bibliography{kuznecov_refs}

\begin{thebibliography}{MMR00}

\bibitem[BGT87]{BinghamGoldieTeugels}
N.~H. Bingham, C.~M. Goldie, and J.~L. Teugels.
\newblock {\em Regular Variation}, volume~27 of {\em Encyclopedia of
  Mathematics and its Applications}.
\newblock Cambridge University Press, 1987.

\bibitem[BGT07]{BurqGerardTzvetkov07}
Nicolas Burq, Patrick G{\'e}rard, and Nikolay Tzvetkov.
\newblock Restrictions of the {L}aplace--{B}eltrami eigenfunctions to
  submanifolds.
\newblock {\em Duke Math. J.}, 136(1):445--486, 2007.

\bibitem[BGV92]{BGV92}
Nicole Berline, Ezra Getzler, and Mich\`ele Vergne.
\newblock {\em Heat Kernels and Dirac Operators}, volume 298 of {\em
  Grundlehren der Mathematischen Wissenschaften}.
\newblock Springer-Verlag, Berlin, 1992.

\bibitem[EP22]{EswarathasanPramanik22}
Suresh Eswarathasan and Malabika Pramanik.
\newblock Restriction of {L}aplace--{B}eltrami eigenfunctions to arbitrary sets
  on manifolds.
\newblock {\em Int. Math. Res. Not. IMRN}, (2):1538--1600, 2022.

\bibitem[Fel71]{Feller1971Vol2}
William Feller.
\newblock {\em An Introduction to Probability Theory and Its Applications,
  Volume II}.
\newblock John Wiley \& Sons, New York, 2 edition, 1971.

\bibitem[GMX24]{GaoMiaoXi24}
Chuanwei Gao, Changxing Miao, and Yakun Xi.
\newblock Refined {$L^p$} restriction estimate for eigenfunctions on
  {R}iemannian surfaces.
\newblock {\em arXiv e-prints}, 2024.
\newblock arXiv:2411.01577.

\bibitem[Gri09]{Grigoryan09}
Alexander Grigor'yan.
\newblock {\em Heat Kernel and Analysis on Manifolds}, volume~47 of {\em AMS/IP
  Studies in Advanced Mathematics}.
\newblock American Mathematical Society, Providence, RI, 2009.

\bibitem[H{\"o}r68]{Hormander1968SpectralFunction}
Lars H{\"o}rmander.
\newblock The spectral function of an elliptic operator.
\newblock {\em Acta Mathematica}, 121:193--218, 1968.

\bibitem[HR03]{HareRoginskaya03}
Kathryn~E. Hare and Maria Roginskaya.
\newblock Energy of measures on compact {R}iemannian manifolds.
\newblock {\em Studia Mathematica}, 159:291--314, 2003.

\bibitem[KWX25]{kaloshin2025kuznecov}
Vadim Kaloshin, Emmett~L. Wyman, and Yakun Xi.
\newblock Kuznecov remainders and generic metrics.
\newblock {\em arXiv preprint arXiv:2507.06887}, 2025.

\bibitem[MMR00]{MattilaMoranRey00}
Pertti Mattila, Manuel Mor{\'a}n, and Jos{\'e}-Manuel Rey.
\newblock Dimension of a measure.
\newblock {\em Studia Mathematica}, 142(3):219--233, 2000.

\bibitem[Pre87]{Preiss87}
David Preiss.
\newblock Geometry of measures in {$\mathbb{R}^n$}: distribution,
  rectifiability, and densities.
\newblock {\em Annals of Mathematics}, 125(3):537--643, 1987.

\bibitem[Sog88]{sogge1988concerning}
Christopher~D Sogge.
\newblock Concerning the ${L}^p$ norm of spectral clusters for second-order
  elliptic operators on compact manifolds.
\newblock {\em Journal of functional analysis}, 77(1):123--138, 1988.

\bibitem[WX23a]{wyman2023can}
Emmett~L. Wyman and Yakun Xi.
\newblock Can you hear your location on a manifold?
\newblock {\em arXiv preprint arXiv:2304.04659}, 2023.

\bibitem[WX23b]{WymanXiKuz}
Emmett~L. Wyman and Yakun Xi.
\newblock A two term {K}uznecov sum formula.
\newblock {\em Communications in Mathematical Physics}, 401:1127--1162, 2023.

\bibitem[Zel92]{Zelditch92}
Steve Zelditch.
\newblock Kuznecov sum formulae and {S}zeg{\H{o}} limit formulae on manifolds.
\newblock {\em Communications in Partial Differential Equations},
  17(1--2):221--260, 1992.

\end{thebibliography}

\end{document}